\documentclass[12pt,a4paper]{amsart}
\usepackage{amssymb,amscd,xypic,amsmath}

\theoremstyle{plain}
\newtheorem{theorem}{Theorem}[section]
\newtheorem{lemma}[theorem]{Lemma}
\newtheorem{proposition}[theorem]{Proposition}
\newtheorem{corollary}[theorem]{Corollary}

\theoremstyle{definition}
\newtheorem{definition}[theorem]{Definition}
\newtheorem{remark}[theorem]{Remark}

\newtheorem{example}[theorem]{Example}

\newcommand\cO{{\mathcal O}}

\newcommand\aff{\operatorname{aff}}
\newcommand\soc{\operatorname{soc}}
\newcommand\uni{\operatorname{uni}}
\newcommand\aq{/_{_{\operatorname{aff}}}\,}
\newcommand\cq{/\hspace*{-2pt}/}

\newcommand\Spec{\operatorname{Spec}}

\title{Observable actions of algebraic groups}

\author{Lex Renner and Alvaro Rittatore}

\thanks{The first named author was partially 
    supported by a grant from NSERC. The second named author was
    artially supported by grants from IMU/CDE    and
NSERC}

\begin{document}

\maketitle

\begin{abstract}
Let $G$ be an affine algebraic group and let $X$ be an affine
algebraic variety.
An action $G\times X \to X$ is
called  {\em observable} if for any $G$-invariant, proper, closed
subset $Y$ of $X$ there is
a nonzero invariant $f\in \Bbbk [X]^G$ such that $f|_{_Y} =0$. We
characterize this condition
geometrically as follows. The action $G\times X \to X$ is observable
if and only
if (1) there is
a nonempty open subset $U\subseteq X$ consisting of closed orbits, and
(2) the field
$\Bbbk (X)^G$ of $G$-invariant rational functions on $X$ is equal to
the quotient field
of $\Bbbk [X]^G$. In case $G$ is reductive, we conclude that there
exists a unique, maximal, $G$-stable,
closed subset $X_{\soc}$ of $X$ such that $G\times X_{\soc} \to X_{\soc}$
is observable. Furthermore, the canonical map $X_{\soc}\cq G \to X\cq G$ is
bijective.

\end{abstract}

\section{Introduction}

A closed subgroup $H$ of the affine algebraic group $G$ is called
an
{\em observable subgroup} if the homogeneous space $G/H$ is a
quasi-affine variety. Such subgroups have been introduced by
Bialynicki-Birula, Hochschild and  Mostow   in \cite{kn:obsdef}, and
researched extensively since then,
notably by F.~Grosshans. See \cite{kn:Gross-14} for a survey on
this topic, and Theorem \ref{thm:obsercond} below for
other useful characterizations of observable subgroups.

In this paper we develop further this point of view, and define the
notion of an \emph{observable action} of $G$ on the affine variety
$X$. Our main results
pertain to actions of non-reductive groups. But we also investigate the
situation where $G$ is reductive.

To state our results we first introduce some notation. Let
$\Bbbk $ be an algebraically closed field. We work
with affine algebraic varieties $X$ over $\Bbbk $. An algebraic
group is assumed to be a smooth, affine, group scheme of finite type
over $\Bbbk $. If $X$ is an affine variety over $\Bbbk $ we
denote by $\Bbbk [X]$ the ring of regular functions on $X$. If
$I\subset \Bbbk[X]$ is an ideal, we denote by $\mathcal
V(I)=\bigl\{x\in X\mathrel{:} f(x)=0 \ \forall \, f\in I\bigr\}$.  If
$Y\subset X$ is a subset, we denote by $\mathcal I(Z)=\bigr\{ f\in
\Bbbk[X]\mathrel{:} f(y)=0\ \forall\, y\in Y\bigr\}$. If $f\in
\Bbbk[X]$, we denote by $X_f=\bigl\{x\in X\mathrel{:} f(x)\neq
0\bigr\}$.  Morphisms
$\varphi:X\to Y$ between affine varieties correspond to morphisms of
algebras $\Bbbk[Y] \to \Bbbk[X]$, by $\varphi\mapsto \varphi^*$,
$\varphi^*(f)=f\circ \varphi$.  If $X$
is irreducible we denote by $\Bbbk (X)$ the field of rational
functions on $X$. If $A$ is any integral domain we denote by $[A]$ its
quotient field. Thus if $X$ is an affine variety, then
$\Bbbk(X)=\bigl[\Bbbk [X]\bigr]$.
If $G\times X\to X$ is a regular action of $G$ on $X$ we consider the
induced action of $G$ on $\Bbbk [X]$, defined as  follows. For $f\in
\Bbbk [X]$ and $g\in G$ we
set $(g\cdot f)(x)=f(g^{-1}x)$. It is well known that $G$-stable
closed subset of $X$ correspond to $G$-stable radical ideals  of
$\Bbbk[X]$. If $x\in X$ we will denote by $\mathcal O(x)=\{g\cdot
x\mathrel{:} g\in G\}$ the \emph{$G$-orbit of $x$}. We say that
$f\in\Bbbk [X]$ is \emph{G-invariant} if $g\cdot f=f$ for any $g\in
G$. The set of $G$-invariants $\Bbbk [X]^G$ forms a
$\Bbbk $-subalgebra of $\Bbbk [X]$, possibly non-finitely
generated.

In this paper we are interested in situations whereby
an affine algebraic group acts on an irreducible, affine variety $X$
in such a  way that
$\bigl[\Bbbk [X]^G\bigr]=\Bbbk (X)^G=\bigl[\Bbbk [X]\bigr]^G$. When this
happens, one can  separate the orbits of maximal dimension {\em generically}
with invariant regular
functions. See for example \cite[Prop.~3.4]{kn:popvin}. In Theorem
\ref{thm:niceopensubset} and Corollary \ref{coro:affinizedsep} we
obtain a generalization of Igusa's criterion (see \cite{kn:igusa} and
\cite[\S 4.5]{kn:popvin}). We express our results in terms of the
\emph{affinized quotient} $\pi: X\to X\aq
G$, induced by the  inclusion $\Bbbk[X]^G\subset \Bbbk[X]$. This
quotient, even when it does not coincide with the
categorical quotient, contains enough information about the action
to give us control of an open subset of orbits of maximal
dimension. See Definition \ref{defi:affquot}.

We define a regular action $G\times X\to X$, to be \emph{observable} if for
any nonzero $G$-stable ideal $I\subset\Bbbk [X]$ we have $I^G\neq (0)$.
A closed subgroup $H$ of $G$ is an observable subgroup, in the sense of Grosshans above,
if and only if the action $H\times G\to G$, by left
translations, is an observable action. It turns out that, in this situation,
the left action is observable if and only if the right action is so.

If  $G\times X\to X$ is an observable action of $G$ on the affine variety $X$,
then it follows easily that $X$ contains a dense open subset consisting entirely
of closed $G$-orbits. This is a sufficient condition for the action to be observable
if $G$ is reductive, but not in general. In Theorem
\ref{thm:obsercharac}, the main result of this paper, we provide the following
geometric characterization. The action $G\times X\to X$ is observable
if and only if (1) $\bigl[\Bbbk [X]^G\bigr]= \bigl[\Bbbk [X]\bigr]^G$
and (2) $X$ contains a dense open subset consisting entirely of closed
$G$-orbits.

It seems natural then to consider the affinized quotient of  an
observable action. In this case, we show that there exists a nonzero invariant
$f\in\Bbbk[X]^G$ such that the affinized quotient $\pi: X_f\to X_f\aq G$ is
a geometric quotient; see Theorem \ref{thm:geoquot}.

As an application,
we consider the case of a reductive group $G$ acting on the
irreducible, affine variety $X$. In this case  we prove that there
exists a unique, maximal, $G$-stable, closed subset $X_{\soc}\subset X$
such that the restricted action $G\times X_{\soc}\to X_{\soc}$ is
observable.  Furthermore, $\Bbbk [X]^G\subseteq\Bbbk [X_{\soc}]^G$ is a
finite, purely inseparable extension. In particular, the morphism
$X_{\soc}\cq G\to
X\cq G$ is finite and bijective. See theorems \ref{thm:socquot} and
\ref{thm:socobser}.

\medskip

{\sc Acknowledgements: } We would like to thank Michel Brion and Walter
Ferrer for their
useful comments and suggestions.

This paper was written during a stay of the second author at the
University of Western Ontario. He would like to thank them for the kind
hospitality he received during his stay.

\section{Preliminaries}

Let $G$ be an affine algebraic group and $X$ an algebraic variety. As
usual, a \emph{(regular) action} of $G$ on $X$
is a morphism $\varphi:G\times X\to X$, denoted by $\varphi(g,x)=g\cdot
x$, such that $(ab)\cdot x= a\cdot (b\cdot x)$ and $1\cdot x=x $ for
all $a,b\in G$ and $x\in X$. Since all the actions we will work with
are regular, we will drop the adjective regular.

\begin{definition} \label{obs.def}

Let $G$ be an affine algebraic group and $H\subset G$ a closed subgroup. The
subgroup {\em $H$ is observable in $G$}\/ if and only if for any
nonzero $H$-stable ideal $I \subset \Bbbk [G]$ we obtain that $I
^H\neq (0)$.
\end{definition}

\begin{remark}
The notion of an observable subgroup was introduced
by  Bialynicki-Birula, Hochschild and  Mostow   in
\cite{kn:obsdef}.
The characterization of observable subgroups
by the condition ``$I^H\neq (0)$'' in Definition \ref{obs.def}
appears for the first time in \cite[Definition 10.2.1]{fer-ritt}.
\end{remark}

\begin{theorem}
\label{thm:obsercond}
Let $G$ be an affine algebraic
group and $H\subset G$ a closed subgroup. Then the following
conditions are equivalent:

\begin{enumerate}
\item The subgroup $H$ is observable in $G$.

\item The homogeneous space $G/H$ is a quasi-affine variety.

\item For an arbitrary proper and closed subset $C \subsetneq
G/H$, there exists an nonzero invariant regular function $ f \in
\Bbbk [G]^H$ such that
$f(C)=0$.
\end{enumerate}
\medskip

Moreover, if  $G$ is connected the above conditions are equivalent to
the condition

\begin{enumerate}
\item[(4)] $\bigl[ \Bbbk [G]\bigr]^H=
\bigl[\Bbbk [G]^H\bigr]$.
\end{enumerate}
\end{theorem}

\begin{proof}
See for example \cite[Thm.~10.5.5]{fer-ritt}.
\end{proof}

\begin{definition}
Let $G$ be an affine algebraic group acting on an affine variety
$X$. We say that  the action is \emph{observable}, if for any nonzero
$G$-stable ideal $I\subset \Bbbk[X] $, $I^G\neq (0)$.
\end{definition}

\begin{example}
(1) Let $G$ be an algebraic group and $H$ a closed subgroup. Then
the action of $H$ on $G$ by translations is observable if and only if
$H$ is observable in $G$.

\noindent (2)  Any action of the unipotent group $U$ on an affine variety is
observable, since for any nonzero $U$-module $M$, we obtain that
$M^U\neq \{0\}$. We will prove in Proposition \ref{prop:unipchar}
below that this
property characterizes unipotent groups.
 \end{example}

First, we recall the construction of the \emph{induced space}:

\begin{definition} \label{induction.def}
Let $ G$ be an algebraic group, and let a closed subgroup $H\subset G$
act on
an algebraic variety $X$. The {\em induced space}\/ $G*_HX$
is defined as the geometric quotient
of $G\times X$ under the $H$-action $h\cdot(g,x)=(gh^{-1},h\cdot
x)$.

Under mild conditions on $X$ (e.g.~$X$ normal and covered by
quasi-projective $H$-stable open subsets), this quotient exists. Clearly,
$G*_HX$ is a $G$-variety, for the action induced by $a\cdot
(g,x)=(ag,x)$. We will denote the class of $(g,x)$ in $G\times X$ by
$[g,x]\in G*_HX$. We refer the
reader to \cite{kn:bb-induced}, where the notion of induced space was
introduced by
Bialynicki-Birula, and to \cite{Ti06}, for a concise survey on this
construction and its many properties, and resume in the next paragraphs
the properties of the induced space we need in what follows.

It is easy to see that $G*_HX$ is a $G$-variety,
for the action $a\cdot [g,x]=[ag,x]$, $a,g\in G$, $x\in X$.

The projection $G\times X\to  G$ induces a morphism $\rho:G*_HX\to
G/H$, $[g,x]\leadsto [g]$. The morphism $\rho$ is a $G$-equivariant fibration, 
with fiber
\[
X=\rho^{-1}([1])=\bigl\{ [1,x]\mathrel{:} x\in X\bigr\}\subset  G*_HX.
\]
Moreover, the $G$-stable closed subsets of $ G*_HX$ correspond
bijectively to
$H$-stable closed subset of $X$, via the maps
\[
 G*_HX\supset Z\leadsto Z\cap
X\quad  \text{and} \quad X\supset Y\leadsto GY\subset  G*_HX.
\]
\end{definition}

\begin{proposition}
\label{prop:unipchar}
Let $G$ be a connected affine algebraic group such that every action
of $G$ on an
affine algebraic variety is observable. Then $G$ is a unipotent group.
\end{proposition}

\begin{proof}
We first observe that every $G$-orbit on an affine $G$-variety $X$ is
closed. Indeed, if $\cO\subset X$ is an orbit, then the action of $G$
on the affine variety $\overline{\cO}$ is observable, and hence $\cO$
is closed.

The result now follows from the fact that if an algebraic
group $G$ acts on any affine algebraic variety with closed orbits,
then $G$ is unipotent (see \cite{fer-uni}). We give a short proof of
this fact, for the sake of completeness.
  Assume  that $G$ contains a multiplicative subgroup
$\Bbbk^*$, and consider the induced space  $G*_{\Bbbk^*}
\Bbbk$, where $\Bbbk^*$ acts on $\Bbbk$ by multiplication. Then
$\cO\bigl([1,1]\bigr)=\bigl\{[g,a]\mathrel{:} g\in G\,,\ a\neq
0\bigr\}$ is an open orbit, and this is a contradiction. It follows
that $G$ is a unipotent algebraic group.
\end{proof}

Let $G$ be an affine algebraic group acting on the affine variety $X$. Then
it is well known that the categorical quotient does not necessarily
exists, even when $\Bbbk[X]^G$ is finitely generated. However, if
$\Bbbk[X]^G$ is finitely generated then $\Spec\bigl(\Bbbk[X]^G\bigr)$ satisfies
a universal property in the category of affine algebraic varieties.

\begin{definition}
\label{defi:affquot}
Let $G$ be an affine algebraic group acting on an affine variety
$X$, in such a way that $\Bbbk[X]^G$ is finitely generated. The
\emph{affinized quotient} of the action is the morphism $\pi:X\to
X\aq G=\Spec\bigl(\Bbbk[X]^G\bigr)$.

It is clear that $\pi$ satisfies the following universal property:

{\em Let $Z$ be an affine variety and $f:X\to Z$  a morphism constant
  on the $G$-orbits. Then
  there exists an unique $\widetilde{f}:X\aq G\to Z$ such that the
  following diagram is commutative.}
\begin{center}
\mbox{
\xymatrix{
X\ar@{^(->}[r]^-f\ar@{->}[d]&Z\\
X\aq G\ar@{^(->}[ur]_-{\widetilde{f}}&
}
}
\end{center}

Indeed, it is clear that the induced morphism $f^*:\Bbbk[Z]\to
\Bbbk[X]$ factors through $\Bbbk[X]^G$.
\end{definition}

Notice that the affinized quotient is the affinization of the
categorical quotient
when the latter exists, and $\Bbbk[X]^G$ is finitely
generated. Indeed, if $X\to X\cq G$ is the affine quotient then
$X\aq G = \operatorname{Spec}\bigl(\mathcal{O}(X\cq G)\bigr)$.

\begin{remark}
\label{rem:quotnotsurj}
It is clear that the morphism $\pi: X\to X\aq G$ is dominant. However, $\pi$ is
not necessarily surjective. Consider a semisimple group $G$
and its maximal unipotent subgroup $U$. Then $G/U=\pi(G/U) $ is known to be
a proper open subset of $G\aq U$.
\end{remark}

Let $G$ be an affine group acting on an affine variety $X$. It is
very easy to prove that  if $f\in \Bbbk[X]^G$, then
$\bigl(\Bbbk[X]^G\bigr)_f\cong \bigl(\Bbbk[X]_f\bigr)^G$. This
technical remark will allow us, combined with Propositions
\ref{prop:obserchar} and \ref{prop:grosshanloc} below, to study affinized
quotients of observable actions.

\begin{lemma}
Let $G$ be an affine algebraic group and $X$ an affine $G$-variety such that
$\Bbbk[X]^G$ is finitely generated. Let $\pi:X\to X\aq G$ be the
affinized quotient. Then for any $f\in \Bbbk[X]^G$, $\pi|_{X_f}:X_f\to
(X\aq G)_f$ is the affinized quotient.
\end{lemma}
\begin{proof}
Indeed, $\bigl(\Bbbk[X]^G\bigr)_f\cong \bigl(\Bbbk[X]_f\bigr)^G$.
\end{proof}

\section{Observable actions and affinized quotients}

In this section we study observable actions. In
Theorem \ref{thm:obsercharac} we
provide a geometric characterization of the observability of an
action.  This characterization allows us to construct an open
subset of $X$
with affine geometric quotient, see Theorem \ref{thm:geoquot}.

We begin by given some equivalent condition for an action to be
observable.

\begin{lemma}
\label{lem:obsred}
Let $G$ be a connected  affine algebraic group acting on an affine
variety $X$,   and let $X=\bigcup_{i=1}^n X_i$ be the
decomposition of $X$ in irreducible
components. Then the action $G\times X\to X$ is observable if  and
only if
the restricted actions $G\times X_i\to X_i$ are observable for all
$i=1,\dots, n$.
\end{lemma}

\begin{proof}
Assume that the action of $G$ on $X$ is observable.
We prove without
  loss of generality that $G\times X_1\to X_1$ is an observable
  action. Let $Y\subsetneq X_1$ be a $G$-stable closed
  subset. Then
  $Z=Y\cup \bigcup_{i=2}^n X_i$ is a $G$-stable closed subset,
  $Z\subsetneq X$. Since the action $G\times X\to X$ is observable,
  there exists $f\in \bigl(\mathcal I
  (Z)^G\setminus \{0\}\bigr)\subset \Bbbk[X]$. It follows that
  $f|_{X_1}\neq
  0$, and hence
  $f\in \bigl(\mathcal I(Y)^G\setminus \{0\}\bigr) \subset
  \Bbbk[X_1]$. Thus,
  the restricted action $G\times X_1\to X_1$ is observable.

For the converse, consider the morphism
$
\pi : \bigsqcup_i X_i\to X$.
Then $\pi$ induces an injective morphism
$A=\Bbbk[X]\to\Pi_i\Bbbk[X_i]=\Pi_iB_i=B$. Let
$I\subseteq\Bbbk[X]$ be a nonzero $G$-ideal. Consider $I_0=\{f\in I\;|\;fB\subseteq A\}$.
One checks easily that $I_0$ is a $G$-ideal of $B$. Now there is a regular element
$r\in A$ such that $A_r\cong B_r = \Pi_i (B_i)_{r_i}$, where $r_i=r|_{X_i}$. One checks that
\[
(I_0)_r = \{f\in I_r\;|\; fB_r\subseteq A_r\} = I_r
\]
since $A_r=B_r$. Thus there exists $m>0$ such that $r^mI\subseteq I_0$. In
particular, $I_0\neq (0)$. But $I_0=\bigoplus_iI_i$,
where $I_i\subseteq\Bbbk[X_i]$ is a $G$-ideal.  Hence
$I_0^G=\bigoplus_iI_i^G\neq (0)$.
But then $I^G\neq (0)$.
\end{proof}

\begin{proposition}
\label{prop:obserchar}
Let $G$ be a connected  affine algebraic group acting on an
 affine
variety $X$. Then the following are equivalent

\begin{enumerate}
\item The action is observable.
\item For any nonempty $G$-stable open subset $U\subset X$, there
  exists $f\in \Bbbk[X]^G\setminus \{0\} $ such that $X_f\subset U$.
\item For every $f\in \Bbbk[X]^G$, the action of $G$ on
  $X_f$ is
  observable.
\item There exists $f\in \Bbbk[X]^G\setminus \{0\}$ such that $f$ is
  not a zero divisor and  that the
  action of $G$ on $X_f$ is
  observable.
\item For every $G$-stable nonzero prime
ideal $\mathcal P$ of $\Bbbk[X]$, $\mathcal P^G\neq (0)$.
\item For every nonempty $G$-stable affine open subset $U$, the
  restriction $G\times U\to U$ is an observable action.
\end{enumerate}
\end{proposition}

\begin{proof}
We prove first the equivalence of conditions (1) -- (6) in the case
when $X$ is an irreducible variety.

The implication (1) $\Longleftrightarrow$ (2) is trivial.

To prove that (1) $\Longrightarrow$ (3),  let $f\in
\Bbbk[X]^G$,   and let $I\subset
\Bbbk[X]_f$ be a nonzero $G$-stable ideal. Then $J=I\cap
\Bbbk[X]\subset \Bbbk[X]$ is a
nonzero $G$-stable ideal. If we let $h\in J^G\setminus\{0\}$
then $h=h/1\in I^G$.

The implication (3) $\Longrightarrow$ (4) is trivial.

In order to prove that (4) $\Longrightarrow$ (5), let $\mathcal
P\subset \Bbbk[X]$ be a nonzero prime
$G$-stable ideal. If $f\in \mathcal P$ there is nothing to prove. If
$f\notin \mathcal P$, let  $\mathcal P_f=\mathcal P\Bbbk [X]_f$ be the
prime ideal generated by $\mathcal P$. Since $f\in \Bbbk[X]^G$, it
follows that $\mathcal P_f$ is $G$-stable; let $h/f^n\in \mathcal
P_f^G\setminus \{0\}$. Then $0\neq h= f^m(h/f^m)\in \bigl( \mathcal P_f\cap
\Bbbk[X]\bigr)^G= \mathcal P^G$.

To prove that (5) $\Longrightarrow$ (1),  let $I\subset
\Bbbk[X]$ be a nonzero
$G$-stable ideal. Then $\mathcal V(I)\subset X$ is a $G$-stable
variety. Since $G$ is connected, every irreducible component of
$\mathcal V(I)$ is $G$-stable, and hence   $r(I)$, the
radical of $I$, is an intersection of $G$-stable prime ideals
$r(I)=\mathcal
P_1\cap\dots \cap \mathcal P_s$. Let $f_i\in \mathcal P_i^G\setminus \{0\}$,
$i=1,\dots, s$. Then $0\neq f=\prod f_i\in r(I)^G$, and thus there
exists $n\geq 0$ such that $0\neq f^n\in I^G$.

In order to prove (1) $\Longrightarrow$ (6),   let $U\subset X$ be a
nonempty $G$-stable
affine open subset. Then by (2) there exists $f\in \Bbbk[X]^G$ such that
$X_f\subset U$. By (4)  it follows that the restriction $G\times
X_f\to X_f$ is observable. Since $U_f=X_f$ it follows that assertion  (4) holds for the action of $G$ on $U$,
i.e.~ that the action on $U$ is observable.

Finally, if (6) holds then the action of $G$ on $X_f$ is observable
for every $f\in \Bbbk[X]^G$. In particular, assertion (3) holds.

Assume now that $X$ is reduced and let $X=\bigcup_i X_i$ be the
decomposition in irreducible components.

The implications (1) $\Longleftrightarrow$ (2) and  (3)
$\Longrightarrow$ (4) are trivial.

To prove that (1) $\Longrightarrow$ (3), just observe that
$X_f=\bigcup_i(X_i)_{f|_{X_i}}$, and apply Lemma \ref{lem:obsred}.

Finally, the remaining implications are proved following the same
reasoning as in the
irreducible case.
\end{proof}

We recall now two useful results that we will need in what
follows.

\begin{proposition}[Grosshans]
\label{prop:grosshanloc}
Let $G$ be an affine algebraic group acting on an affine variety. Then
there exists $f\in \Bbbk[X]^G\setminus \{0\}$ such that
$\bigl(\Bbbk[X]_f\bigr)^G\cong \bigl(\Bbbk[X]^G\bigr)_f$ is
finitely generated.
\end{proposition}

\begin{proof}
See \cite[Thm.~1]{kn:grosshanslocal}.
\end{proof}

\begin{theorem}[Rosenlicht \cite{Ro56, Ro63}]
\label{thm:rosenlicht}
Let $G$ be an algebraic group and $X$  an irreducible
 $G$-variety. Then there exists a $G$-stable  open
subset $\emptyset\neq X_0\subset X$ such that
 the action of $G$ restricted to $X_0$ has a  geometric quotient.
\end{theorem}
\begin{proof}
See for example \cite[\S 13.5]{fer-ritt}.
\end{proof}

If $G\times X\to X$ is an observable action then the field of
invariant rational functions is generated by invariant regular
functions (see Theorem \ref{thm:obsercharac} below).
In \cite{kn:popvin} those authors provide some useful sufficient
conditions
to guarantee that $\bigl[\Bbbk[X]\bigr]^G=\bigl[\Bbbk[X]^G\bigr]$.

The following theorem may be regarded as a generalization of ``Igusa's
criterion''. See \cite{kn:igusa} and
\cite[\S 4.5]{kn:popvin}.

\begin{theorem} \label{popvin.thm}
Let $G$ be an algebraic group acting on an affine algebraic variety
$X$. Assume that either

\noindent (a) $G^0$, the connected component of $1$, is solvable, or

\noindent (b) $\Bbbk[X]$ is factorial.

Then every rational function $f\in\bigl[\Bbbk[X]\bigr]^G$ is the
quotient of two semiinvariant regular functions. In
other words, there exist  $h,g\in \Bbbk[X]$ and $\chi \in \mathcal
X(G)$   such that $f=\frac{g}{h}$,  $a\cdot g=\chi(a) g$ and $a\cdot
h=\chi(a)h$ for all $a\in G$.

Moreover, if $G$ has no nontrivial character (that is, if
$\Bbbk[G]^* =\Bbbk^*$), then $\bigl[\Bbbk[X]\bigr]^G=\bigl[\Bbbk[X]^G\bigr]$.
\end{theorem}

\begin{proof}
See \cite[Thm.~3.3]{kn:popvin}.
\end{proof}

\begin{definition}
It is well known that the set of points whose orbit has maximal
dimension is an open subset, we denote it by
\[
X_{\max}=\bigl\{ x\in X\mathrel{:} \dim \cO(x)\geq \dim \cO(y) \
\forall \, y\in X\bigr\}.
\]
We let
\[
\Omega(X)=\bigl\{ x\in X\mathrel{:} \dim \cO(x) \text{ is maximal
and } \overline{\cO(x)}=\cO(x)\bigr\}.
\]
\end{definition}

\begin{theorem}
\label{thm:niceopensubset}
Let $G$ be an affine algebraic group acting on an affine variety $X$,
such that $\Bbbk[X]^G$ is finitely generated, and let
$\pi:X\to X\aq G$ be the affinized quotient. Assume  that

(1)  $\bigl[\Bbbk[X]^G\bigr]= \bigl[\Bbbk[X]\bigr]^G$.

Then

(2) There exists a nonempty open subset $W\subset X\aq G$ such that
  $\pi^{-1}(x)=\overline{\cO(y)}$, with $y\in X_{\max}$, for all $x\in
  W$. Moreover, $W$ can be taken of the form $(X\aq G)_f=X_f\aq G$, with $f\in
\Bbbk[X]^G$.

Conversely, if (2) holds then  $\bigl[\Bbbk[X]^G\bigr]\subset
\bigl[\Bbbk[X]\bigr]^G$ is a purely inseparable finite extension. In
particular, if $\operatorname{char} \Bbbk=0$, then
$\bigl[\Bbbk[X]^G\bigr]= \bigl[\Bbbk[X]\bigr]^G$.
\end{theorem}

\begin{proof}
In order to prove the implication (1) $\Longrightarrow$ (2), let
$X_0\subset X$ be as in Rosenlicht's
Theorem \ref{thm:rosenlicht}, and $\rho:X_0\to X_0/G$ the geometric
quotient.
Then  we have the following commutative diagram
\begin{equation}
\label{eqn:thediag}
\xymatrix{
X_0\ar@{^(->}[r]\ar@{->>}[d]_-{\rho}&X\ar@{->}[d]^-{\pi}\\
X_0/G\ar@{->}[r]_-{\varphi}& X\aq G
}
\end{equation}
Since
\[
\begin{split}
\Bbbk(X_0/G)\cong \Bbbk(X_0)^G=&\   \Bbbk(X)^G=\bigl[\Bbbk[X]\bigr]^G=
\bigl[\Bbbk[X]^G\bigr]=\\
& \
  \bigl[\Bbbk[X\aq G]\bigr] =\Bbbk(X\aq G),
\end{split}
\]
it follows that the
geometric quotient $X_0/G$ and the affinized quotient
$X\aq G$ are birationally equivalent via $\varphi$. Hence, there exists
an open subset $V\subset X_0/G$ such that
$\varphi|_{_V} : V\rightarrow X\aq G$ is an open immersion.  Replacing
$X_0$ by
$U=\pi^{-1}(V)$ we can assume that $\varphi$ is an open immersion, and
hence $\rho=\pi|_{X_0}$.

Let $V\subset X\aq G$ be an open subset such that, for each $w\in V$, the fiber
$\pi^{-1}(w)$ has all irreducible components of maximal
dimension  (as follows  from Chevalley's Theorem,
see for example
\cite[Thm.~1.5.4]{fer-ritt}).
Let $m=\dim X-\dim X\aq G$ be the  dimension of such a fiber.
 It follows from diagram \eqref{eqn:thediag} that $m$ is
the maximal dimension of the orbits of $G$ on $X$. Moreover,   if
$w\in W=(X_0/G)\cap V\subset X\aq G$, then
there
exists an unique irreducible component $Z\subset \pi^{-1}(w)$ such
that $Z=\overline{\cO(x)}$, with $x\in X_0$.   Let $Y\subset
  \pi^{-1}(w)$ another irreducible
  component; then $Y\subset X\setminus X_0$.
Let $A\subset W$ be the set of point which fiber is reducible,
$A=\bigl\{ w\in W\mathrel{:} \pi^{-1}(w)\neq \overline{\pi^{-1}(w)\cap
  X_0}\bigr\}$.
Since  $A\subset \pi(X\setminus X_0)$, it follows that if $A$ is dense in
$X_0$, then $X\setminus X_0$ dominates $X\aq G$. Hence the maximal dimension
of a fiber of $\pi|_{_{X\setminus X_0}}$ is $\dim (X\setminus X_0)-\dim X\aq G<
\dim X-\dim X\aq G=m$. But for $a\in A$, $\dim \bigl(\pi|_{_{X\setminus
  X_0}}^{-1}(a)\bigr)=m$, which is a contradiction. Hence, for all  $w\in
W\setminus
\overline {A}\neq \emptyset$, we have that
$\pi^{-1}(w)=\overline{\cO(x)}$, with $x\in X_0$.

Let now $f\in \Bbbk[X]^G$ such that  $(X\aq G)_f\subset W\setminus
\overline{A}$. Then $X_f$
satisfies assertion (2).

Assume now that (2) holds, with $W=(X\aq G)_f$, $f\in
\Bbbk[X]^G$. Then $\pi^{-1}(W)=X_f$,  $\bigl[ \Bbbk[X]^G\bigr]=\bigl[
\Bbbk[X]_f^G\bigr]= \bigl[ \Bbbk[X_f]^G\bigr]$ and $\bigl[
\Bbbk[X]\bigr]^G= \bigl[ \Bbbk[X]_f\bigr]^G=\bigl[
\Bbbk[X_f]\bigr]^G$. Thus, we can assume without loss of generality
that $X$ is such that $\pi^{-1}(y)=\overline{G\cdot x}$, with $x\in
X_{\max}$, for all $y\in X\aq G$. Let $X_0\subset X$ be as in
Rosenlicht's Theorem and consider the commutative diagram
\begin{center}
\mbox{\xymatrix{
X_0\ar@{^(->}[r]\ar@{->>}[d]_-{\rho}&X\ar@{->}[d]^-{\pi}\\
X_0/G\ar@{->}[r]_-{\varphi}& X\aq G
}}
\end{center}

Let $y,z\in X_0/G$ be such that  $\varphi(y)=\varphi(z)$. Then
$\rho^{-1}(y)$ and $\rho^{-1}(z)$ are contained in
$\pi^{-1}\bigl(\varphi(y)\bigr)$, and hence
$\rho^{-1}(y)=\rho^{-1}(z)$. It follows that $\varphi$ is
injective. Since $\varphi$ is dominant, it follows that  $\varphi$
induces a finite,
purely inseparable  extension $\bigl[\Bbbk[X]^G\bigr]\subset
\bigl[\Bbbk[X]\bigr]^G$.
\end{proof}

\begin{remark}
(1) Observe that if $G$ is reductive, then $\Bbbk[X]^G$ is finitely
generated, and thus Theorem \ref{thm:niceopensubset} applies.

(2) Let $G\times X\to X$ an action satisfying condition (1) of Theorem
\ref{thm:niceopensubset}, and $\mathcal O, \mathcal O'\subset
X_{\max}\cap X_f$, with $f$ as in condition (2) of that theorem. Then
there exists $h\in \Bbbk[X_f\aq G]=\Bbbk[X]_f^G$ such that
$f\bigl(\pi(\mathcal O)\bigr)   =0$ and $f\bigl(\pi(\mathcal
O')\bigr)=1$. Hence $\Bbbk[X]_f^G$, or equivalently the affinized
quotient $\pi: X_f\to X_f\aq G$, separates orbits in $X_{\max}\cap
X_f$.
\end{remark}

\begin{corollary}
\label{coro:affinizedsep}
Let $G$ be an algebraic group acting on an affine algebraic variety, in
such a way that $ \bigl[\Bbbk[X]^G\bigr]= \bigl[\Bbbk[X]\bigr]^G$. Then
there exists $f\in \Bbbk[X]^G$ such that $\Bbbk[X]_f^G$ separates the
orbits in $X_{\max}\cap X_f$.
\end{corollary}

\begin{proof}
By Proposition \ref{prop:grosshanloc}, there exists  $h\in \Bbbk[X]^G$
such that $\Bbbk[X_h]=\Bbbk[X]_h$ is finitely generated. Since
\[
\bigl[\Bbbk[X_h]^G\bigr]=\bigl[\Bbbk[X]^G\bigr]=\bigl[\Bbbk[X_h]\bigr]^G
\]
it follows
from Theorem \ref{thm:niceopensubset} that there exists $g\in
\Bbbk[X_h]^G=\Bbbk[X]_{h}^G$ such that
$\Bbbk\bigl[(X_h)_g\bigr]^G=\Bbbk[X]_{gh}^G$ separates the orbits in
$X_{\max}\cap X_{hg}$.
\end{proof}

We now characterize observability in terms of the underlying geometry.

\begin{theorem}
\label{thm:obsercharac}
Let $G$ be an affine group acting on an affine variety $X$. Then the
action is observable if and only if
(1) $\bigl[\Bbbk[X]^G\bigr]=\bigl[\Bbbk[X]\bigr]^G$ and (2) $\Omega(X)$
contains a nonempty open subset.
\end{theorem}

\begin{proof}
Assume that the action is observable and let
 $Y=X\setminus X_{\max}$ and let $ f\in \mathcal I (Y)^G\setminus
\{0\}$. Then the  nonempty affine open
subset $X_f\subset X_{\max}$ is $G$-stable and any point $x\in X_f$ is
such that $\dim \mathcal O(x)$ is maximal. Let $\cO\subset X_f$ be an
orbit. Since every
orbit is open in its closure, it follows that if
$\overline{\cO}\neq \cO$ then $\overline{\cO}\cap Y\neq
\emptyset$. Since $f$ is constant on orbit closures, it follows that
$f|_{_{\cO}}=0$ and hence we obtain a contradiction.
It
follows that $\emptyset\neq X_f\subset \Omega(X)$.

Clearly $\bigl[\Bbbk[X]^G\bigr]\subset \bigl[\Bbbk[X]\bigr]^G $. Let $g\in
\bigl[\Bbbk[X]\bigr]^G$, and consider the ideal $I=\bigl\{ f\in
\Bbbk[X] \mathrel{:} fg\in \Bbbk[X]\bigr\}$. Clearly $I$ is $G$-invariant,
and hence there exists $ f\in \Bbbk[X]^G$ such that $fg\in\Bbbk[X]^G$.

In order to prove the converse, let $f\in \Bbbk[X]^G$ such that
$\Bbbk[X_f]^G$ is finitely generated (see Proposition
\ref{prop:grosshanloc}). Then, by Proposition \ref{prop:obserchar},
the action of $G$ on $X$ is observable
if and only if the action on $X_f$ is so. Thus, we can assume without
loss of generality that  $\Bbbk[X]^G$ is finitely generated. Let  $\pi:X\to
X\aq  G$ be the affinized quotient. By
Theorem \ref{thm:niceopensubset}, there exists
$f\in \Bbbk[X]^G$ such that $\pi^{-1}(y)=\overline{G\cdot x}$ for all
$y\in V=(X\aq G)_f\cong X_f\aq  G$. Moreover, we have the
following commutative diagram
\begin{center}
\mbox{
\xymatrix{
(X_f)_0\, \ar@{^(->}[r]\ar@{->>}[d]_-{\rho}&
X_f\ar@{->}[d]^-{\pi}\ar@{^(->}[r]&
X\ar@{->}[d]^-{\pi}\\
(X_f)_0/G\, \ar@{^(->}[r]_-\varphi& X_f\aq G\ar@{^(->}[r]&
X\aq G
}
}
\end{center}

\noindent where $\rho$ is a geometric quotient. Since
$\Bbbk\bigl((X_f)_0/G\bigr)= \Bbbk\bigl((X_f)_0\bigr)^G  =
\Bbbk(X_f)^G$, it follows
by hypothesis that $\Bbbk\bigl((X_f)_0/G\bigr)
=\Bbbk(X_f\aq G)$. Since $\rho$
and $\pi$ separate closed orbits, it follows that $\varphi$ is an open
immersion.

Since $\Omega(X)$ contains a nonempty open subset, it follows that
$\Omega(X)\cap (X_f)_0\neq
\emptyset$. Let $g\in \Bbbk[X]^G$ such that $(X_f\aq G)_g\subset
(X_f)_0/G$. If $y\in X_f\aq G)_g$, then
$\pi^{-1}(y)=\overline{\mathcal O(x)}$, where $x\in \Omega(X)\cap (X_f)_0$,
hence $\pi^{-1}(y)$ is a closed
orbit of maximal dimension. Therefore, $\pi|_{_{X_{fg}}}: X_{fg}\to
(X_f\aq G)_g\cong X_{fg}\aq G$ is such that all its fibers are
closed orbits. Replacing $X$ by $X_{fg}$, we can hence assume that all
the fibers of the
affinized quotient are closed orbits. In this situation, it is clear
that any nonzero $G$-stable ideal $I\subset \Bbbk[X]$ is such that
$I ^G\neq (0)$. Indeed, if $Y\subset X$ is a $G$-stable closed
subset such that dominates $X\aq G$, then $\pi(Y)$ contains an
open subset of $X\aq G$ and hence $Y=\pi^{-1}\pi(Y)$ contains an
open subset of $X$; that is $Y=X$. It follows that if $Y\subsetneq X$ is
a proper $G$-stable closed subset, then there  exists $z\in
(X/_{\aff}G)\setminus \overline{\pi(Y)}$. Let $f\in \Bbbk\bigl[
X\aq G\bigr]=\Bbbk[X]^G$ such that $f(z)=1$ and
$f\bigl(\overline{\pi(Y)}\bigr)=0$. Then $f\in \mathcal
I(Y)^G\setminus \{0\}$.
\end{proof}

The following two examples show that, in Theorem \ref{thm:obsercharac},
both conditions
\begin{enumerate}
\item $\bigl[\Bbbk[X]^G\bigr]=\bigl[\Bbbk[X]^G\bigr]$, and
\item $\Omega(X)$ contains a nonempty open subset
\end{enumerate}
are necessary.

\begin{example}
(1) Consider the action of $\operatorname{GL}_n(\Bbbk)$ on
$\operatorname{M}_{n}(\Bbbk)$ by left translations. Then
 $\bigl[\Bbbk[\operatorname{M}_{n}(\Bbbk)]\bigr]^{\operatorname{GL}_n(\Bbbk)}
 =
 \bigl[\Bbbk[\operatorname{M}_{n}(\Bbbk)]^{\operatorname{GL}_n(\Bbbk)}\bigr]$
 while we
have that $\Omega\bigr(\operatorname{M}_{n}(\Bbbk)\bigr)=\emptyset$.

\medskip

\noindent (2) Let $G$ be a semisimple group,  $B\subset G$ a Borel
subgroup, and consider the $B$-action on $G$ by left
translations. Then $ \Omega(G)=G$ while
$\bigl[\Bbbk[G]^B\bigr]=\Bbbk\subsetneq \bigl[\Bbbk[G]\bigr]^B$.
\end{example}

\begin{remark}
Observe that it follows from the proof of  Theorem \ref{thm:obsercharac},
that for any $f\in \Bbbk[X]^G$ such that $X_f\subset X_{\max}$, then
$X_f\subset\Omega(X)$.
\end{remark}

With further assumptions, on either the algebraic group $G$ or the
affine variety $X$, there is a simplified characterization of the 
observability of the action $G\times X\to X$.

\begin{proposition}
\label{prop:partialconv}
Let $G=L\times U$ be a direct product of a reductive group with a
unipotent  group, and let $X$ be a affine
$G$-variety such that $\Omega(X)$ contains a nonempty open
subset. Then  $\bigl[\Bbbk[X]^G\bigr]=
\bigl[\Bbbk[X]\bigr]^G$; that is, the action $G\times X\to X$ is
observable.
\end{proposition}

\begin{proof}
Let $I\subsetneq \Bbbk[X]$ be a nonzero $G$-stable ideal. Then
$\mathcal V(I) \neq \emptyset $ is a proper closed subset. Since there
exists a open subset of  closed orbits, then there exists a closed orbit
$Z$ such that $Z\cap \mathcal V(I)=\emptyset$. Since $L$ is reductive
and that $Z$ is $L$-stable,
it follows that there exists $f\in \Bbbk[X]^L$ such that $f\in I$,
$f|_{_{Z}}=1$. Hence, $I^L\neq (0)$. Since $U$ normalizes $L$,
it follows that $I\cap \Bbbk[X]^L\neq (0)$ is an $U$-submodule, and hence
$I^G=(I^L)^U\neq (0)$.
\end{proof}

\begin{definition}
Let $\mu : G\times X\to X$ be an action of $G$ on the  affine
variety $X$. We denote
\[
E_G(X) = \bigl\{\chi\in\mathcal{X}(G)\mathrel{:}\text{there is a
  nonzero $\chi$-semiinvariant }
       f\in\Bbbk[X]\bigr\}.
\]
\end{definition}

\begin{corollary}
\label{fact.cor}
Let $\mu : G\times X\to X$ be an action of $G$ on the factorial affine
variety $X$.
Then the following are equivalent.
\begin{enumerate}
\item $\mu$ is observable.
\item \begin{enumerate}
       \item $\Omega(X)$ contains a dense, open subset of closed $G$-orbits.
       \item $E_G(X) $ is a group.
      \end{enumerate}
\end{enumerate}
\end{corollary}
\begin{proof}
Assume that (1) holds. Then by Theorem~\ref{thm:obsercharac},
$\Omega(X)$ contains a dense, open subset of closed orbits. Let $f$ be
a nonzero semiinvariant
for $\chi\in \mathcal X(G)$. If $\chi\in E_G(X)$ is not a unit in
this monoid then the nonzero $G$-ideal
$f\Bbbk[X]\subseteq\Bbbk[X]$ has no nonzero $G$-invariant, and thus $\mu$ is not observable.

Conversely, assume that (2) holds, and let $I$ be a nonzero $G$-ideal of $\Bbbk[X]$. Let $G_{\uni}$
be the closed normal subgroup of $G$ generated by its unipotent
elements.
Since $G_{\uni}$ is normal in $G$, any closed $G$-orbit of $X$ consists
of closed $G_{\uni}$-orbits.
Thus, by Theorem~\ref{popvin.thm} and Theorem~\ref{thm:obsercharac},
$J=I^{G_{\uni}}$ is a nonzero ideal of $\Bbbk[X]^{G_{\uni}}$.  Now
$G_{\uni}$ acts trivially on $J$, so that $G$ acts on $J$ through the
torus $S=G/G_{\uni}$. Let $f\in J$ be a nonzero $\chi$-semiinvariant
for this action. Since $E_G(X)$
is a group there is a nonzero $\chi^{-1}$-semiinvariant $g\in
\Bbbk[X]$. Thus $fg\in I\cap \Bbbk[X]^{G}=I^G$ is the
desired invariant.
\end{proof}

\begin{corollary} \label{solve.cor}
Let $\mu : G\times X\to X$ be an action of $G$ on the affine variety $X$
and assume that $G$ is solvable. Then the following are equivalent.
\begin{enumerate}
\item $\mu$ is observable.
\item $E_G(X) $ is a group.
\end{enumerate}
\end{corollary}
\begin{proof}
The proof is similar to the proof of Corollary~\ref{fact.cor}, taking
into account
the fact that, in this case, $G_{\uni}$ is a unipotent group, which is
always observable.
\end{proof}

The next result shows that observability is well behaved under
{\em induction} of actions (see Definition \ref{induction.def}).

\begin{proposition}
Let $G$ be an affine algebraic group and $H\subset G$ a closed
subgroup, such that $G/H$ is affine. Let $H\times X\to X$ be an
action of $H$ on the affine variety $X$. Then the action of $H$ is
observable if and only if the induced
action of $G$ on the affine variety $G*_HX$ is observable.
\end{proposition}

\begin{proof}
First of all, observe that since $G/H$ is affine, it follows that $G*
_HX$ is an affine $G$-variety affine. Hence,
$\Bbbk[G*_HX]=\bigl(\Bbbk[G]\otimes
\Bbbk[X]\bigr)^H$, where if $f\otimes g\in \Bbbk[G]\otimes
\Bbbk[X]$, then $\bigl(a\cdot (f\otimes
g)\bigr)(b,x)=f(ba)g(a^{-1}\cdot x)$, for all $a\in H$, $b\in G$,
$x\in X$. A simple calculation shows then that $\Bbbk[G*_HX]^G=\Bbbk\otimes
\Bbbk[X]^H\cong \Bbbk[X]^H$.

Assume now that the action $H\times X\to X$ is observable and let
$Z\subset G*_HX$ be  a nonempty $G$-stable closed subset. If
$\rho:G*_HX\to G/H$ is the canonical projection, then $Z\cap
\rho^{-1}(1H)\neq \emptyset$ is a $H$-stable closed subset. It follows
that there exists $f\in
\mathcal I\bigl(Z\cap \rho^{-1}(0)\bigr)^H\setminus \{0\}$. Then
$1\otimes f\in \mathcal I(Z)^G\setminus \{0\}$.

Assume now that the action $G\to G*_HX$ is observable, and let
$Y\subset X$ be a $H$-stable closed subset. Then $GY\subset G*_HX$ is
a $G$-stable closed subset, and hence there exists $1\otimes f\in
\Bbbk\otimes \Bbbk[X]^H= \Bbbk[G*_HX]^G$ such that $(1\otimes
f)(Z)=0$. It follows that $f(Y)=f\bigl(Z\cap \rho^{-1}(0)\bigr)=0$.
\end{proof}

As the final result of this section we show that, if the action is
observable, there is a refinement in the construction of the open set 
$X_0\subseteq X$ of Rosenlicht's Theorem \ref{thm:rosenlicht}. 
In this case, there exists a $G$-stable \emph{principal} affine open subset 
$X_f\subseteq X$ with affine geometric quotient.

\begin{theorem}
\label{thm:geoquot}
Let  $G$ an affine algebraic group acting on  an affine variety $X$,
such that the action is observable. Then there exists $f\in
\Bbbk[X]^G\setminus \{0\}$ such that
$\Bbbk[X]^G_f$ is a finitely
generated $\Bbbk$-algebra and $X_f\to X_f\aq G$ is a geometric quotient.
\end{theorem}

\begin{proof}
By Proposition \ref{prop:grosshanloc}, there exists $f\in
\Bbbk[X]^G\setminus \{0\}$ such that
$\Bbbk[X]^G_f$ is a finitely
generated algebra. Let $X_0\subset X$ be an open subset as in Rosenlicht's
Theorem. Then  by Proposition \ref{prop:obserchar} there exists $g \in
\Bbbk[X]^G\setminus \{0\}$ such that $X_g\subset X_0$. It follows that
$X_{fg}\subset X_0$ is such that $\bigl(\Bbbk[X]_{fg}\bigr)^G\cong
\bigl(\Bbbk[X]_{f}^G\bigr)_g$ is finitely generated.

Let $h=fg$ and consider the commutative diagram
\begin{center}
\mbox{
\xymatrix{
X_h\ar@{->}[r]^-{\pi}\ar@{->}[d]_-{\rho}&X_h\aq G\\
X_h/G \ar@{->}[ur]_-\varphi&
}
}
\end{center}

Then,
 \[
\Bbbk(X_h/G)\cong
\Bbbk(X_h)^G=\bigl[\Bbbk[X_h]\bigr]^G\cong  \bigl[\Bbbk[X]_h\bigr]^G=
\bigl[\Bbbk[X]_h^G\bigr].
\]

 It follows that $\varphi$ is a
birational morphism.

Let $q\in \Bbbk[X]_f^G$ such that $\varphi$ restricted
to $\varphi^{-1}\bigl((X_h\aq G)_q\bigr)$ is an isomorphism
over $(X_h\aq G)_q$. Since
$\pi^{-1}\bigl((X_h\aq G\bigr)_q\bigr)= X_{hq}$, it
follows that  $\rho|_{_{X_{hq}}}:X_{hq} \to
\varphi^{-1}\bigl((X_h\aq G)_q\bigr)$ is the geometric quotient.
\end{proof}

\section{The reductive case}

In what follows we apply the results of the previous section to the situation
of a reductive group $G$ acting on the affine variety $X$.
 We prove
 in Theorem \ref{thm:socobser} that there exists a unique, maximal, $G$-stable
 closed subset $X_{\soc}\subset X$ such that the restricted action
 $G\times X_{\soc} \to X_{\soc}$ is observable. Furthermore, the
 categorical quotient $X_{\soc}\cq G$ exists. See Theorem
 \ref{thm:socquot}.

\begin{definition}
Let $G$ be an affine algebraic group acting on an algebraic variety $X$. We
define the \emph{socle} of the action to be
\[
X_{\soc}=\overline{\bigcup_{\overline{\mathcal O(x) }=\mathcal O(x) }
  \cO(x)}.
\]
\end{definition}

\begin{remark}
Observe that $(X_{\soc})_{\soc}=X_{\soc}$.
\end{remark}

\begin{lemma}
 Let $G$ be a reductive group acting on an affine variety
 $X$, and $\pi:X\to X\cq G$ the categorical quotient. Then
\[
\Omega(X)= \pi^{-1}\bigl( (X\cq G)\setminus \pi(X\setminus X_{\max})\bigr).
\]

In particular, $\Omega(X)$ is a (possibly empty) $G$-stable open subset.
\end{lemma}

\begin{proof}
It is clear that $\Omega(X)\subset \pi^{-1}\bigl( (X\cq G)\setminus
\pi(X\setminus X_{\max})\bigr)$.  Since
$X\setminus X_{\max}$ is a $G$-stable closed subset, it follows
that $\pi(X\setminus X_{\max})$ is closed in $X\cq G$. Let $x\in
X\setminus \Omega(X)$; then $\overline{\cO(x)}$ contains an unique closed
orbit $\cO'$, with $\dim \cO'<\dim \cO$. It follows that
 $\pi(x)=\pi(\cO')\in \pi (X\setminus X_{\max})$; that is $x\notin
 \pi^{-1}\bigl( (X\cq G)\setminus \pi(X\setminus X_{\max})\bigr)$.
\end{proof}

\begin{proposition}[{Popov, \cite{Pop}}]
\label{prop:equivstability}
Let $G$ be reductive group acting on an  affine algebraic
variety $X$, is such a way that $\Omega(X)\neq
\emptyset$. Then $\Omega(X)$ is an open subset of $X$.
\end{proposition}
\begin{proof}
See \cite[Theorem 4]{Pop} or \cite[Theorem~13.3.13]{fer-ritt}.
\end{proof}

Next we rephrase the contents of Proposition \ref{prop:equivstability}
in a obvious but useful characterization of condition (2) of Theorem
\ref{thm:obsercharac}, namely that $\Omega(X)$ contains a nonempty
open subset.

\begin{lemma}
\label{lem:socleandomega}
Let $G$ be a reductive group acting on an affine variety $X$, and
$\pi : X\to X\cq G$ the categorical quotient. Then the
following are equivalent.

\begin{enumerate}
\item $\pi|_{_{X\setminus X_{\max}}}: X\setminus X_{\max}\to X\cq G$ is
  dominant.

\item $\pi|_{_{X\setminus X_{\max}}}: X\setminus X_{\max}\to X\cq G$ is
  surjective.

\item $X_{\max}$ does not contain any closed orbit.

\item $\Omega(X)=\emptyset$.
\qed
\end{enumerate}
\end{lemma}

\begin{theorem}
\label{thm:socquot}
Let $G$ be a reductive group acting on an affine variety $X$, and
$Z\subset X$ a $G$-stable closed subset such that $\pi|_{_Z}:Z\to X\cq G$
is dominant. Then the unique morphism $\varphi:Z\cq G\to
X\cq G$, that fits into the following commutative diagram,
\begin{center}
\mbox{
\xymatrix{
Z\ar@{^(->}[r]\ar@{->>}[d]_-\rho&X\ar@{->>}[d]^-\pi\\
Z\cq G\ar@{->}[r]_-{\varphi} & X\cq G
}
}
\end{center}
is a bijective, purely inseparable
morphism. That is, $\varphi$ is a bijective morphism, and   $\varphi^*:
\Bbbk[X]^G\to \Bbbk[Z]^G$
induces  a purely inseparable field extension
$\bigl[\Bbbk[X]^G\bigr]\subset \bigl[\Bbbk[Z]^G\bigr]$.

Moreover, $X_{\soc}\subset Z$,  $X_{\soc}$ is irreducible,
and $\Omega(X_{\soc})\neq \emptyset$.

In particular, if $\operatorname{char} \Bbbk =0$, then $\pi|_{_Z}:Z\to
X\cq G$ is the categorical quotient and $\varphi : X_{\soc}\cq G\to X\cq G$ is an
isomorphism.
\end{theorem}

\begin{proof}
Since $Z\subset X$ is closed, it follows that $\pi(Z)\subset X\cq G$ is
closed, and
hence $\pi(Z)=X\cq G$. In particular, if $\cO\subset X$ is a closed
orbit, it follows that $\pi(\cO)\in \pi(Z)=X$, and hence there exists an
orbit $\cO'\in Z$ such that $\pi(\cO')=\pi(\cO)$. Therefore, since
$\pi$ separates closed orbits, it follows that $\cO'=\cO \subset Z$
and hence $X_{\soc}\subset Z$.

Since $\varphi\circ \rho$ is dominant, it follows that $\varphi$ is
dominant. Since $\pi$ and $\rho$ separate closed orbits, then
$\varphi$ is injective. Thus, we have the
 following commutative
 diagram:
\begin{center}
\mbox{
\xymatrix{
\Bbbk[X]^G\ar@{^(->}[r]\ar@{^(->}[rrd]_-{\varphi^*}&\Bbbk[X]\ar@{->>}[r]&
\Bbbk[Z]\\
& & \Bbbk[Z]^G\ar@{^(->}[u]
}
}
\end{center}
Since $G$ is reductive,  $\Bbbk[Z]^G$ is finitely generated and there
exists $n\geq 0$ such that for all $f\in
\Bbbk[Z]^G$, there exists $g\in \Bbbk[X]^G$, with
$\varphi(g)=f^{p^n}$ (see for example \cite[\S 9.2]{fer-ritt}).  Hence
$\pi|_{_Z}:Z\to X\cq G$ is a purely
inseparable morphism. It follows that if $\operatorname{char}\Bbbk=0$,
then $X\cq G\cong Z\cq G$.

Since $\pi|_{_{X_{\soc}}}:X_{\soc}\to X\cq G$ is surjective, it follows
that $ \varphi: X_{\soc}\cq G\to X\cq G$ is  a purely inseparable
morphism.

Let $Z\subset X_{\soc}$ an irreducible component such that $\pi|_{_Z}$ is
dominant. Then $X_{\soc}\subset
Z$, and hence $X_{\soc}$ is irreducible.

Consider $Z=X_{\soc}\setminus (X_{\soc})_{\max}$. If $Z$ dominates
$X_{\soc}\cq G$, it follows that $X_{\soc}=(X_{\soc})_{\soc}\subset Z$, that is a
contradiction. Hence, by Lemma \ref{lem:socleandomega}, it follows
that $\Omega(X_{\soc})\neq  \emptyset$.
\end{proof}

\begin{theorem}
\label{thm:socobser}
Let $G$ be reductive group acting on an  affine algebraic
variety $X$.  Then the action is observable if and only if $\Omega(X)\neq
\emptyset$. In particular, $X_{\soc}$ is the maximum $G$-stable
closed
subset $Z\subset X$ such that the restricted action $G\times Z\to Z$
is observable.
\end{theorem}

\begin{proof}
If the action is observable, it follows from Theorem
\ref{thm:obsercharac} that $\Omega(X) \neq \emptyset$.
Assume now that $\Omega(X)\neq \emptyset $ and let $Z\subsetneq X$ be a
$G$-stable closed subset. If $\Omega(X)\subset Z$ it follows that
$Z=X$; hence $\Omega(X)\setminus  Z\neq \emptyset$.
Recall that the categorical
quotient $\pi:X\to X\cq G=\Spec\bigl(\Bbbk[X]^G\bigr)$  separates closed
orbits.  It follows that
$ \Omega(X) \setminus \pi^{-1}\bigl(\pi(Z)\bigr)\neq \emptyset$, since
the closed orbits belonging to $Z$ and $\pi^{-1}\bigl(\pi(Z)\bigr)$
are the same. Let $\cO\subset  \Omega(X) \setminus Z$ be a closed
orbit. Then $\pi^{-1}\bigl(\pi(\cO)\bigr)=\cO$, again because $\pi$
separates closed orbits. Since $\pi$ also separates saturated closed
subsets, it follows  that there exists $f\in \Bbbk[X]^G$ such that
$f\in \mathcal I\bigl(\pi^{-1}\bigl(\pi(Z)\bigr)\bigr)\subset \mathcal
I(Z)$ and $f(\cO)=1$;
in particular, $f\in \mathcal
I(Z)^G\setminus \{0\}$.

It follows by the very definition of $X_{\soc}$  (e.g.~from
Proposition \ref{prop:equivstability})  that
$\Omega(X_{\soc})\neq \emptyset$. Let now $Z$ be a $G$-stable
irreducible closed
subset such that the restricted action is observable; then $\Omega(Z)$
is a nonempty open subset of $Z$, consisting of closed orbits in $Z$,
and hence in $X$. It follows that $Z=\overline{\Omega(Z)}\subset
X_{\soc}$.  If $Y$ is any $G$-stable closed subset, it follows by
Lemma \ref{lem:obsred} that the restriction of the action
to any irreducible component $Z$ is
observable, and hence $Y\subset X_{\soc}$.
\end{proof}

\begin{proposition}
Let $G$ be a reductive group acting on an affine variety $X$. Then
$\mathcal I(X_{\soc})$ is the maximum $G$-stable ideal such that
$I^G=(0)$.
\end{proposition}

\begin{proof}
Let $I=\sum_{ J^G=(0)}J$ be the sum of all $G$-stable ideals
such that $J^G=(0)$, and consider the canonical $G$-morphism
$\varphi: \bigoplus_{ J^G=(0)}J\to I$. Since $\varphi$ is
surjective, it follows that for every $f\in I^G$ there exist $n\geq 0$
and $h\in
\bigl(\bigoplus_{J^G=(0)}J\bigr)^G$ such that
$\varphi(h)=f^{p^n}$, where $\operatorname{char}\Bbbk=p$.
But $\bigl(\bigoplus_{J^G=(0)}J\bigr)^G=(0)$.
Hence, $I^G=(0)$.

Let $\cO\subset X$ be a closed orbit and assume that $\cO\cap \mathcal
V(I)=\emptyset$. Since $\Bbbk[X]^G$ separates $G$-stable closed
subsets, if follows that there exists $f\in \Bbbk[X]^G$ such that
$f|_{_\cO}=1$ and $f|_{_{\mathcal V(I)}}=0$, hence $I^G\neq (0)$ and
we get a contradiction. Therefore, $X_{\soc}\subset \mathcal V(I)$.

Observe that if $f\in r(I)^G$  is such that $f^n\in I$, it follows
that for any $a\in G$, then  $a\cdot (f^n)=f^n\in I$, and hence $f=0$.
 Thus, $r(I)^G=(0)$ and by maximality  then $I=r(I)$.
By  Theorem \ref{thm:socobser}, if we prove that the action
$G\times \mathcal V (I)\to \mathcal V(I)$ is observable, then
$X_{\soc}=\mathcal V(I)$. But $\Bbbk\bigl[\mathcal V(I)\bigr]\cong
\Bbbk[X]/I$, and hence  $G$-stable ideals of
$\Bbbk\bigl[\mathcal V(I)\bigr]$ are of the form $J/I$, were
$J\subset \Bbbk[X]$ is a $G$-stable ideal containing $I$. Then if
$J/I\neq (0)$ it follows that $I\subsetneq J$ and hence, by maximality
of $I$,
$J^G\neq (0)$. Thus, $(J/I)^G\neq (0)$, since $\Bbbk[X]^G$ injects in
$\Bbbk[X]/I$.
\end{proof}

\bigskip 

\bigskip
\begin{scriptsize}
\noindent \begin{tabular}{ll}
{\sc Lex Renner}   \vspace*{2pt} &          {\sc Alvaro Rittatore} \\ 
University of Western Ontario \hspace*{5.5cm} &    Facultad de Ciencias\\
London, N6A 5B7,  Canada        &  Universidad de la Rep\'ublica\\ 
{\tt lex@uwo.ca}               &  Igu\'a 4225\\
                &      11400 Montevideo, Uruguay\\
 & {\tt alvaro@cmat.edu.uy}
\end{tabular}
\end{scriptsize}

\end{document}